\documentclass[a4paper,12pt]{amsart}

\usepackage{amssymb,amsmath,amsthm,latexsym}
\usepackage{amsfonts}
	\usepackage{amsfonts}
\usepackage{graphicx}
\usepackage[pdftex,bookmarks,colorlinks=false]{hyperref}
\usepackage{tikz}
\usepackage{verbatim}
\usepackage{caption}
\usepackage{subcaption}
\newtheorem{theorem}{Theorem}[section]
\newtheorem{lemma}[theorem]{Lemma}

\newtheorem{proposition}[theorem]{Proposition}
\newtheorem{definition}[theorem]{Definition}

\newtheorem{problem}[theorem]{Problem}
\newtheorem{remark}[theorem]{Remark}

\begin{document}

\title[On the Sparing Number of Certain Graph Structures]{ON THE SPARING NUMBER OF CERTAIN GRAPH STRUCTURES}

\author{N K Sudev}
\address{Department of Mathematics, Vidya Academy of Science \& Technology, Thalakkottukara P O, Thrissur, Kerala, India.}
\email{sudevnk@gmail.com}

\author{K A Germina}
\address{Department of Mathematics, School of Mathematical \& Physical Sciences, Central University of Kerala, Kasaragod, Kerala, India.}
\email{srgerminaka@gmail.com}

%%%%%%%%%%%%%%%%%%%%%%%%%%%%%%%%%%%%%%%%%%%%%%%%%%%%%%%%%%%%%%%%%%%

\begin{abstract}
An integer additive set-indexer is defined as an injective function $f:V(G)\rightarrow 2^{\mathbb{N}_0}$ such that the induced function $g_f:E(G) \rightarrow 2^{\mathbb{N}_0}$ defined by $g_f (uv) = f(u)+ f(v)$ is also injective. An IASI $f$ is said to be a weak IASI if $|g_f(uv)|=max(|f(u)|,|f(v)|)$ for all $u,v\in V(G)$. A graph which admits a weak IASI may be called a weak IASI graph. The Sparing number of a graph $G$ is the minimum number of edges, with the singleton set-label, required for $G$ to admit a weak IASI.  In this paper, we study about the sparing number of certain graph structures and provide some results on them.

\vspace{2mm}


\vspace{2mm}


\end{abstract}

%\thanks{}

%%%%%%%%%%%%%%%%%%%%%%%%%%%%%%%%%%%%%%%%%%%%%%%%%%%%%%%%%%%%%%%%%%%

\maketitle

%%%%%%%%%%%%%%%%%%%%%%%%%%%%%%%%%%%%%%%%%%%%%%%%%%%%%%%%%%%%%%%%%%%

\section {Introduction}

%%%%%%%%%%%%%%%%%%%%%%%%%%%%%%%%%%%%%%%%%%%%%%%%%%%%%%%%%%%%%%%%%%%

For all  terms and definitions, not defined specifically in this paper, we refer to \cite{FH}. Unless mentioned otherwise, all graphs considered here are simple, finite and have no isolated vertices.

Let $\mathbb{N}_0$ denote the set of all non-negative integers. For all $A, B \subseteq \mathbb{N}_0$, the sum of these sets is denoted by  $A+B$ and is defined by $A + B = \{a+b: a \in A, b \in B\}$. The set $A+B$ is called the sumset of the sets $A$ and $B$. 

An {\em integer additive set-indexer} (IASI, in short) is defined in \cite{GA} as an injective function $f:V(G)\rightarrow 2^{\mathbb{N}_0}$ such that the induced function $g_f:E(G) \rightarrow 2^{\mathbb{N}_0}$ defined by $g_f (uv) = f(u)+ f(v)$ is also injective.

\cite{GS1} The cardinality of the labeling set of an element (vertex or edge) of a graph $G$ is called the {\em set-indexing number} of that element.

\begin{lemma}\label{L-Card}
\cite{GS1} Let $A$ and $B$ be two non-empty finite subsets of $\mathbb{N}_0$. Then $max(|A|,\\|B|) \le |A+B|\le |A|.|B|$. Therefore, for an integer additive set-indexer $f$ of a graph $G$, we have $max(|f(u)|, |f(v)|)\le |g_f(uv)|= |f(u)+f(v)| \le |f(u)| |f(v)|$, where $u,v\in V(G)$.
\end{lemma}

\begin{definition}{\rm
\cite{GS1} An IASI $f$ is called a {\em weak IASI} if $|g_f(uv)|=max(|f(u)|,\\|f(v)|)$ for all $u,v\in V(G)$. A graph which admits a weak IASI may be called a {\em weak IASI graph}. A weak  IASI $f$ is said to be {\em weakly $k$-uniform IASI} if $|g_f(uv)|=k$, for all $u,v\in V(G)$ and for some positive integer $k$.}
\end{definition}

\begin{definition}{\rm
\cite{GS3} An element (a vertex or an edge) of graph which has the set-indexing number 1 is called a {\em mono-indexed element} of that graph.}
\end{definition}

\begin{definition}{\rm
\cite{GS3} The {\em sparing number} of a graph $G$ is defined to be the minimum number of mono-indexed edges required for $G$ to admit a weak IASI and is denoted by $\varphi(G)$.}
\end{definition}

\begin{theorem}\label{T-WSG}
\cite{GS3} If a graph $G$ is a weak IASI graph, then any subgraph $H$ of $G$ is also a weak IASI graph or equivalently, If $G$ is a graph which has no weak IASI, then any supergraph of $G$ does not have a weak IASI.
\end{theorem}

\begin{theorem}\label{T-WBP}
\cite{GS3} All bipartite graphs admit a weak IASI. Hence, all paths, trees and even cycles admit a weak IASI. We also, observe that the sparing number of bipartite graphs is $0$.
\end{theorem}

\begin{theorem}\label{T-SB1}
\cite{GS3} If a connected graph $G$ admits a weak IASI, then $G$ is bipartite or $G$ has at least one mono-indexed edge. 
\end{theorem}

\begin{theorem}\label{T-WKN}
\cite{GS3} The complete graph $K_n$ admits a weak IASI if and only if the number of edges of $K_n$ that have set-indexing number 1 is $\frac{1}{2}(n-1)(n-2)$.
\end{theorem}

\begin{theorem}\label{T-WUOC}
\cite{GS3} An odd cycle $C_n$ has a weak IASI if and only if it has at least one mono-indexed edge. That is, the sparing number of an odd cycle is $1$.
\end{theorem}

\begin{theorem}\label{T-NME}
\cite{GS3} Let $C_n$ be a cycle of length $n$ which admits a weak IASI, for a positive integer $n$. Then, $C_n$ has an odd number of mono-indexed edges when it is an odd cycle and has even number of mono-indexed edges, when it is an even cycle. 
\end{theorem}

\begin{theorem}\label{T-SNUG}
\cite{GS4} Let $G_1$ and $G_2$ be two weak IASI graphs. Then, $\varphi(G_1\cup G_2) = \varphi(G_1)+\varphi(G_2)-\varphi(G_1\cap G_2)$.
\end{theorem}

%%%%%%%%%%%%%%%%%%%%%%%%%%%%%%%%%%%%%%%%%%%%%%%%%%%%%%%%%%%%%%%%%%%

\section {Sparing Number of Graph Union}

%%%%%%%%%%%%%%%%%%%%%%%%%%%%%%%%%%%%%%%%%%%%%%%%%%%%%%%%%%%%%%%%%%%

In this paper, we discuss about the sparing number of different graph operations. All sets under consideration in this paper are subsets of $\mathbb{N}_0$.

\begin{proposition}\label{T-SNUP}
Let $P_m$ and $P_n$ be two paths. Then, the sparing number of $P_m\cup P_n$ is either $0$ or $1$. 
\end{proposition}
\begin{proof}
{\bf Case-1:} If $P_m$ and $P_n$ have exactly one common end vertex, then $P_m\cup P_n$ is also a path. Hence, by Theorem \ref{T-WBP}, the sparing number of  $P_m\cup P_n$ is $0$.

Let both the end vertices of $P_m$ and $P_n$ are the same. Then $P_m\cup P_n$ is a cycle. Therefore, we have the following situations.

\noindent {\bf Case-2:} Let $P_m$ and $P_n$ are of same parity. Then, $P_m\cup P_n$ is an even cycle. Hence, by Theorem \ref{T-WBP}, the sparing number of $P_m\cup P_n$ is $0$.

\noindent {\bf Case-3:} Let $P_m$ and $P_n$ are of different parity. Then, $P_m\cup P_n$ is an odd cycle. Hence, by Theorem \ref{T-NME}, the sparing number of $P_m\cup P_n$ is $1$.
\end{proof}

\begin{lemma}\label{T-SNUCP}
Let $C_n$ be a cycle that has a weak IASI and $P_m$ be a path. Then, the sparing number of $C_n\cup P_m$ is $0$ or $1$.
\end{lemma}

\begin{proof}
{\bf Case-1:} Let $C_n$ and $P_m$ have one point in common. If $C_n$ is an even cycle, then $\varphi(C_n)=0$ and $\varphi(P_m)=0$. Hence, by Theorem \ref{T-SNUG}, $\varphi(C_n\cup P_m)=0$.

Let both the end points of $P_m$ lie on the cycle $C_n$. Let $u$ and $v$ be the end vertices of $P$ which lie on the cycle $C_n$. Let $C'$ and $C''$ the two sections of the cycle $C_m$ which have the common end points $u$ and $v$.

\noindent {\bf Case-2:} Let $C_n$ and $P_m$  are of even length. Clearly, $C'$ and $C''$ are simultaneously odd or simultaneously even. 

First, assume that $C'$ and $C''$ be of even length. Then, the cycles $C'\cup P_m$ and $C''\cup P_m$ are even cycles and have the sparing number $0$. Therefore, $\varphi(C_n\cup P_m)=0$. 

Next, assume that $C'$ and $C''$ be of odd length. Then, the cycles $C'\cup P_m$ and $C''\cup P_m$ are odd cycles and have the sparing number $1$. Since $C_n$ is even, the sections $C'$ and $C''$ need not have any mono-indexed edges. Hence, $P$ has the required mono-indexed edge.  Therefore, $\varphi(C\cup P)=1$.

\noindent {\bf Case-3:} Let $C_n$ be an odd cycle and $P_m$ is a path of even length. Then, either $C'$ or $C''$ is odd and the other is even. Without loss of generality, $C'$ is odd and $C''$ is even. Then, $C'\cup P_m$ is an odd cycle and $C''\cup P_m$ is an even cycle. Therefore, the sparing number of $C'\cup P_m$ is $1$ and that of $C''\cup P_m$ is $0$. Therefore, the section $C'$ has a mono-indexed edge. Hence, $\varphi(C_m\cup P_n)=1$

\noindent {\bf Case-4:} Let $C_n$ is an even cycle and $P_m$  is a path of odd length. 

First, assume that $C'$ and $C''$ be of odd length. Then, the cycles $C'\cup P_m$ and $C''\cup P_m$ are even cycles and have the sparing number $0$. Therefore, $\varphi (C_n\cup P_m)=0$. 

Next, assume that $C'$ and $C''$ be of even length. Then, the cycles $C'\cup P_m$ and $C''\cup P_m$ are odd cycles and have the sparing number $1$. Since $C_n$ is even, the sections $C'$ and $C''$ need not have any mono-indexed edges. Hence, $P$ has the required mono-indexed edge.  Therefore, $\varphi(C_n\cup P_m)=1$. 

\noindent {\bf Case-5:} Let $C_n$ and $P_m$  are of odd length. Then, either $C'$ or $C''$ is odd and the other is even. Without loss of generality, $C'$ is odd and $C''$ is even. Then, $C'\cup P_m$ is even and $C''\cup P_m$ is odd. Then, there is no  number of mono-indexed edges in $C'\cup P_m$ and the number of mono-indexed edge in $C''\cup P_m$ is $1$. Therefore, the section $C''$ has one mono-indexed edge. Hence, $\varphi(C_n\cup P_n)=1$.
\end{proof}

\begin{theorem}\label{T-SNUC}
Let $C_m$ and $C_n$ be two cycles which admit a weak IASI. Then, the sparing number of $C_m\cup C_n$ is either $0$, $1$ or $2$. 
\end{theorem}
\begin{proof}
{\bf Case-1:} Let $C_m$ and $C_n$ be even cycles. By Theorem, \ref{T-WBP}, the sparing number of both $C_m$ and $C_n$ is $0$. Hence, by Theorem \ref{T-SNUG}, $\varphi(C_m\cup C_n) = 0$.

\noindent {\bf Case-2:} Without loss of generality, let $C_m$ be an odd cycle and $C_n$ be an even cycle. Then, by Theorem \ref{T-NME}, the sparing number $C_m$ is 1 and sparing number of $C_n$ is zero. Label the end vertices of an edge, say $e$, of $C_m$ that is not common to $C_n$ by distinct singleton sets so that $e$ is mono-indexed and label the succeeding vertices by distinct non-singleton sets and distinct singleton sets alternately. Also, label the vertices of $C_n$ alternately by distinct singleton sets and non-singleton sets. Clearly, this labeling is a weak IASI for $C_m\cup C_n$. Hence, $\varphi(C_m)=1$ and $\varphi(C_n)=0$ and $\varphi(C_m\cap C_n)=0$. Then, by Theorem \ref{T-SNUG}, $\varphi(C_m\cup C_n)=1$.

\noindent {\bf Case-3:} Let $C_m$ and $C_n$ be odd cycles with some edges in common. Then, by Theorem \ref{T-NME}, $C_m$ and $C_n$ have sparing number $1$. Label the end vertices of an edge, say $e$, that is not common to both $C_m$ and $C_n$ by distinct singleton sets so that $e$ is mono-indexed and label the succeeding vertices by distinct non-singleton sets and distinct singleton sets alternately. This labeling is a weak IASI for $C_m\cup C_n$. Hence, $\varphi(C_m)=1$ and $\varphi(C_n)=1$ and $\varphi(C_m\cap C_n)=1$. Then, by Theorem \ref{T-SNUG}, $\varphi(C_m\cup C_n)=1$.

\noindent {\bf Case-4:} Let $C_m$ and $C_n$ be edge disjoint odd cycles. Then, by Theorem \ref{T-NME}, $C_m$ and $C_n$ have sparing number $1$. Hence, $\varphi(C_m)=1$ and $\varphi(C_n)=1$ and $\varphi(C_m\cap C_n)=0$. Then, by Theorem \ref{T-SNUG}, $\varphi(C_m\cup C_n)=2$.
\end{proof}

In the following discussions in this section, the graphs $G$ that are finite union of cycles are only considered. 

Let the graph has edge disjoint odd cycles. Let $G=C_{m_1}\cup C_{m_2}\cup C_{m_3}\cup \ldots \cup C_{m_n}$. Without loss of generality, let $C_{m_1},C_{m_2},\ldots, C_{m_r}$ be the odd cycles and $C_{m_{r+1}}, C_{m_{r+2}},\ldots,C_{m_n}$ be the even cycles in $G$. We proceed to find the sparing number in several cases as follows.

Suppose that all the cycles in $G$ are edge disjoint. We know that an odd cycle has the sparing number $1$ and an even cycle has the sparing number $0$. In view of these facts, recall the following theorem on Eulerian graphs.

\begin{theorem}\cite{CZ}\label{T-EG}
A graph $G$ is an Eulerian graph if and only if it can be decomposed into a finite number of edge disjoint cycles.
\end{theorem}

The following theorem is a consequence of Theorem \ref{T-EG}.

\begin{theorem}\label{T-WEG}
An Eulerian graph $G$ has a weak IASI if and only if at least $r$ mono-indexed edges of $G$, where $r$ is the number of odd cycles in $G$.
\end{theorem}
\begin{proof}
Let $G$ be an Eulerian graph. Since $G$ is Eulerian, by Theorem \ref{T-EG}, $G$ can be decomposed into edge disjoint cycles. Let $C_{m_1},C_{m_2},C_{m_3},\ldots, C_{m_n}$ be the edge disjoint cycles in $G$. Without loss of generality, let $C_{m_1},C_{m_2},C_{m_3},\\ \ldots, C_{m_r}$ be the odd cycles and $C_{m_{r+1}}, C_{m_{r+2}},\ldots, C_{m_{n}}$ be the even cycles in $G$.

Now, assume that $G$ admits a weak IASI, each cycle in $G$ must have a weak IASI, say $f_i$, where $f_i$ is the restriction of $f$ to the cycle $C_{m_i}$.  Even cycles, which admit weak IASI, do not necessarily have a  mono-indexed edge. By Theorem \ref{T-WUOC}, each odd cycles $C_i$, which admits a weak IASI, say $f_i$, must have at least one edge of set-indexing number $1$. Therefore, $G$ has at least $r$ mono-indexed edges.

Conversely, assume that $G$ has $r$ mono-indexed edges. Label the vertices of $G$ in such a way that one edge each in every odd cycle $C_i,1\le i\le r$ is mono-indexed. Then, by Theorem \ref{T-WUOC}, each odd cycle $C_{m_i}, 1\le i\le r$ has a weak IASI, say $f_i$. Also, we note that every even cycle $C_{m_j}, r+1\le j\le m$ can have a weak IASI, say $g_j$. Define $f$ on $V(G)$ by
\[ f(v) = \left\{
  \begin{array}{l l}
    f_i(v) & \quad \text{if $v$ is in the odd cycle $C_{m_i}$}\\
    g_j(v) & \quad \text{if $v$ is in the even cycle $C_{m_j}$}
  \end{array} \right.\]
Clearly, $f$ is a weak IASI of $G$.
\end{proof}

From the above theorem, we observe that the sparing number of an Eulerian graph $G$ is $r$, where $r$ is the number of odd cycles in $G$. Figure \ref{EG-3OC} depicts a weak IASI for an Eulerian graph having three edge disjoint odd cycles.

\begin{figure}[h!]
\centering
\includegraphics[width=0.85\textwidth]{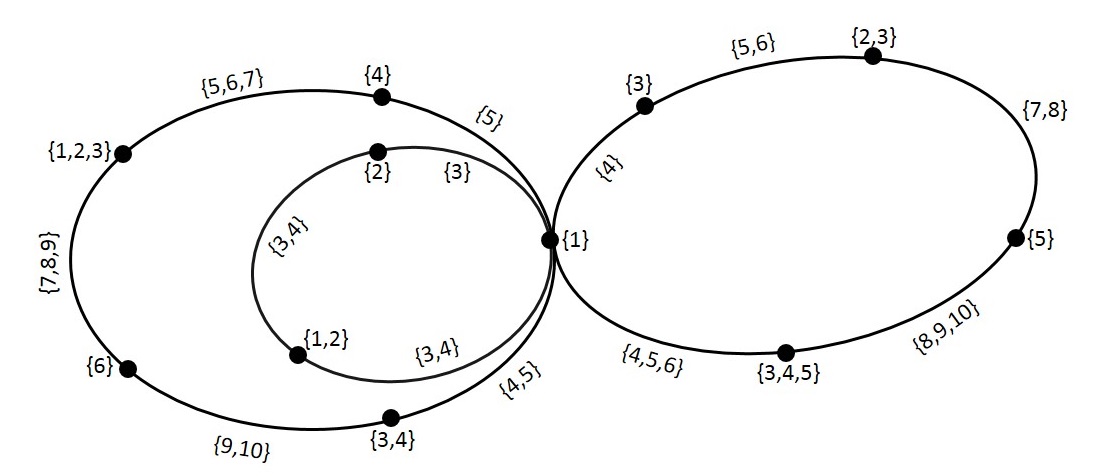}
\caption{}\label{EG-3OC}
\end{figure}

In the following section, we focus on the graphs which have cycles having some edges in common.

%%%%%%%%%%%%%%%%%%%%%%%%%%%%%%%%%%%%%%%%%%%%%%%%%%%%%%%%%%%%%%%%%%%

\section {Sparing Number of Conjoined Graphs}

%%%%%%%%%%%%%%%%%%%%%%%%%%%%%%%%%%%%%%%%%%%%%%%%%%%%%%%%%%%%%%%%%%%

\begin{definition}{\rm
If a set of cycles in a graph $G$ has the same common edge (or path), then that set of cycles is called {\em conjoined structure} of cycles in $G$. If $G$ itself is a conjoined structure of cycles, then it is called a {\em conjoined graph}.}
\end{definition}

If all cycles in a conjoined graph $G$ are odd cycles, then $G$ is called an {\em odd-conjoined graph.} and if all cycles in a conjoined graph $G$ is are even cycles, then $G$ is called an {\em even-conjoined graph}.

Figure \ref{G-CG-E} depicts an even conjoined graph and Figure \ref{G-CG-O} depicts an odd conjoined graph.

\begin{figure}[h!]
\begin{center}
\begin{subfigure}[b]{0.4\textwidth}
\includegraphics[scale=0.35]{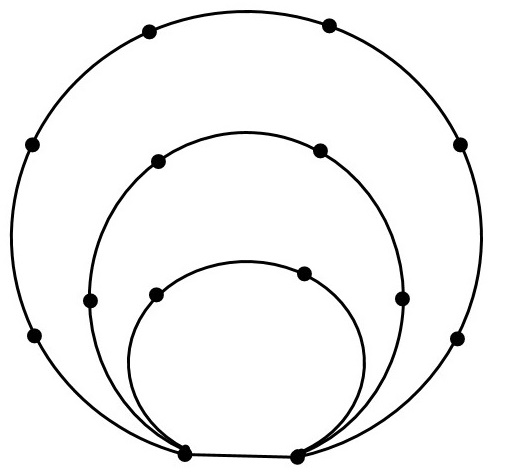}
\caption{\small \sl Even-Conjoined Graph.}\label{G-CG-E}
\end{subfigure}
\quad
\begin{subfigure}[b]{0.4\textwidth}
\includegraphics[scale=0.35]{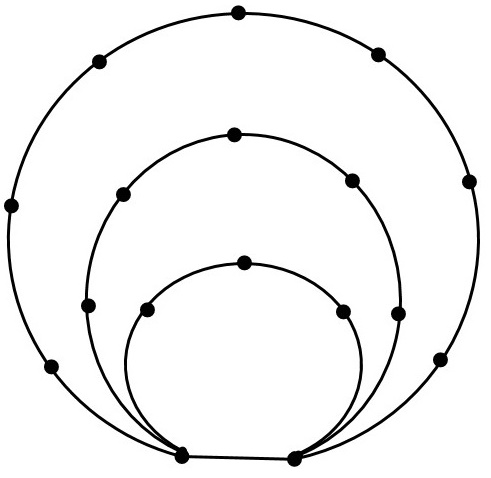}
\caption{\small \sl Odd-Conjoined Graph.}\label{G-CG-O}
\end{subfigure}
\caption{}
\end{center}  
\end{figure}

\begin{theorem}
The sparing number of an odd-conjoined graph is $1$ and that of an even-conjoined graph is $0$.
\end{theorem}
\begin{proof}
Let $G=\displaystyle{\bigcup_{i=1}^n}C_{m_i}$ be an odd conjoined graph. Then, each $C_{m_i}$ is an odd cycle and hence by Theorem \ref{T-NME}, $C_{m_i}$ must have at least one mono-indexed edge. Label the vertices of $G$ in such a way that a common edge of all cycles is mono-indexed. Clearly, this labeling is a weak IASI of $G$. Therefore, the sparing number of $G$ is $1$.

If $G$ is an even-conjoined graph, then each $C_{m_i}$ is even. Hence, by Theorem \ref{T-WBP}, each $C_{m_i}$ has the sparing number $0$. Therefore, sparing number of $G$ is $0$.
\end{proof}

\begin{remark}{\rm
We note that an odd-conjoined graph is analogous to an odd cycle, since the sparing numbers of both graphs are $1$ and both graphs contain odd number number of mono-indexed edges.}
\end{remark}

\begin{theorem}\label{T-SN-CCP}
Let the graph $G$ contains $r$ odd cycles and $l$ even cycles all of them have a common path. Then, the sparing number of $G$  is given by
\[ \varphi(G) = \left\{
  \begin{array}{l l}
    r & \quad \text{if $r\le l$}\\
    l+1& \quad \text{if $l<r$}
  \end{array} \right.\]
\end{theorem}
\begin{proof}
Let $G$ be the given graph with all cycles in it have a common path. Assume that $G$ admits a weak IASI $f$ . Then, every cycle in $G$ also admits a weak IASI. All odd cycles in $G$ must have at least one mono-indexed edge.

\noindent {\bf Case-1:} Let $r\le l$. Then, with respect to the function $f$, label the vertices of $G$ by distinct singleton sets and distinct non-singleton sets in such a way that one edge in each odd cycle, not common to any other cycle, is mono-indexed and no edge of even cycles is mono-indexed. Then, the sparing number of $G$ is $r$. 

\noindent {\bf Case-2:} Let $l<r$. Then, with respect to the function $f$, label the vertices of $G$ by distinct singleton sets and distinct non-singleton sets in such a way that one edge in the path common to all cycles is mono-indexed. Now, all the odd cycles have a mono-indexed edge. But, since even cycles have even number of mono-indexed edges, every even cycle contains another mono-indexed edge under $f$. Therefore, $G$ has at least $l+1$ mono-indexed edges. Hence, the sparing number of $G$ is $l+1$.
\end{proof}

%%%%%%%%%%%%%%%%%%%%%%%%%%%%%%%%%%%%%%%%%%%%%%%%%%%%%%%%%%%%%%%%

%%%%%%%%%%%%%%%%%%%%%%%%%%%%%%%%%%%%%%%%%%%%%%%%%%%%%%%%%%%%%%%%%%%

\section {Sparing Number of Entwined Graphs}

%%%%%%%%%%%%%%%%%%%%%%%%%%%%%%%%%%%%%%%%%%%%%%%%%%%%%%%%%%%%%%%%%%%

\begin{definition}{\rm
Let every path of $G$ appears in at most two cycles and every cycle is attached to at most two cycles. Then exactly two cycles have a common path. Such a graph $G$ is called {\em entwined graph of cycles} in $G$. If $G$ itself is a entwined graph of cycles, then it is called an {\em entwined graph}.

A entwined graph that contains only odd cycles is called an {\em odd-entwined graph} and a entwined graph that contains only even cycles is called an {\em even-entwined graph}.}
\end{definition}

Figure \ref{G-SNCH-EC} depicts an odd-entwined graph with even number of cycles and Figure \ref{G-SNCH-OC} depicts an odd-entwined graph with odd number of cycles.

\begin{figure}[h!]
\begin{center}
\begin{subfigure}[b]{0.45\textwidth}
\includegraphics[scale=0.35]{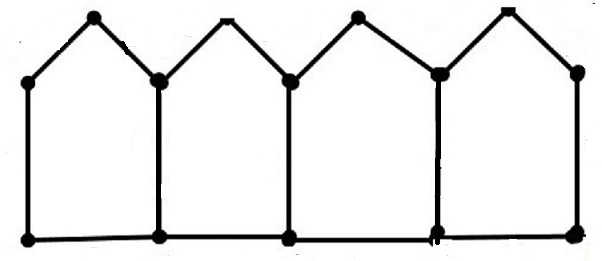}
\caption{\small \sl An odd-entwined graph with even number of cycles .}\label{G-SNCH-EC}
\end{subfigure}
\quad
\begin{subfigure}[b]{0.45\textwidth}
\includegraphics[scale=0.35]{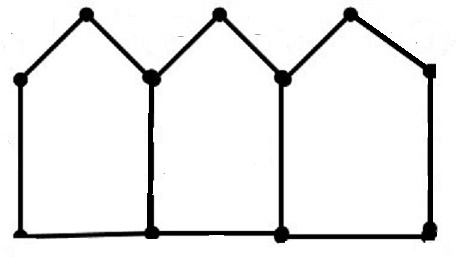}
\caption{\small \sl An odd-entwined graph with odd number of cycles.}\label{G-SNCH-OC}
\end{subfigure}
\caption{}
\end{center}  
\end{figure}

\begin{theorem}\label{T-SNOCG}
Let $G$ be an odd-entwined graph. Then, $G$ admits a weak IASI if and only if it has at least $\lfloor \frac{(n+1)}{2}\rfloor$ mono-indexed edges.
\end{theorem}
\begin{proof}
For $1<j\le \lfloor \frac{(n+1)}{2}\rfloor$, label the vertices of the cycles $C_{m_{2j-1}}$ and $C_{m_{2j}}$ with distinct singleton sets and distinct non-singleton sets alternately, so that a common edge of $C_{m_{2j-1}}$ and $C_{m_{2j}}$ is mono-indexed. Here, if $G$ contains even number of cycles, then $G$ must have $\frac{n}{2}$ mono-indexed edges and if $G$ contains odd number of cycles, then $G$ must have $\frac{n+1}{2}$ mono-indexed edges. Hence, the sparing number of an odd-entwined graph on $n$ cycles is $\lfloor \frac{(n+1)}{2}\rfloor$.
\end{proof}

We observe that if $G$ is an even-entwined graph, then its sparing number is $0$. The sparing number of entwined graphs, in which cycles of different parities share common edges, is discussed in Theorem \ref{T-SNEG1}.

\begin{theorem}\label{T-SNEG1}
Let $G$ be a entwined graph such that all odd cycles in $G$ are edge disjoint. Then, $G$ admits a weak IASI if and only if it has at least $r$ mono-indexed edges, where $r$ is the number of odd cycles in $G$.
\end{theorem}
\begin{proof}
Let $G$ be a entwined graph in which the odd cycles $C_{m_1},C_{m_2},\ldots,C_{m_r}$ are edge disjoint. If two cycles in $G$ have a common edge, then they are of different parity or both of them are even cycles. Assume that $G$ is a weak IASI graph. Then, each $C_{m_i}$ is a weak IASI graph. By Theorem, each odd cycle $C_{m_i}$ has at least one mono-indexed edge. Therefore, $G$ has at least $r$ mono-indexed edges. 

Conversely, assume that $G$ can have (at least) $r$ mono-indexed edges. Label the vertices of the odd cycles $C_{m_i}, 1\le i\le r$ alternately by distinct singleton sets and distinct non-singleton sets so that at least one edge of each odd cycle $C_{m_i}$ is mono-indexed. Clearly, this labeling is a set-indexer and hence is a weak IASI for $G$.
\end{proof}

\begin{remark}{\rm
The sparing number of a entwined graph $G$ with edge disjoint odd cycles is the number of odd cycles in $G$.}
\end{remark}

%%%%%%%%%%%%%%%%%%%%%%%%%%%%%%%%%%%%%%%%%%%%%%%%%%%%%%%%%%%%%%%%%%%%%%%%%%%%%%%%%%%%%%%%%%%%%%%%%%%%%%%%%%%%%%%%%%%%%%%%%%%%%%%%%%%

\section {Sparing Number of Floral Graphs}

%%%%%%%%%%%%%%%%%%%%%%%%%%%%%%%%%%%%%%%%%%%%%%%%%%%%%%%%%%%%%%%%%%%
All finite union of cycles need not have entwined structures. Some cycles may have common edges with more than two cycles. We call such a structure as {\em floral graphs}.

\begin{definition}{\rm
Let $G$ be a graph in which some or all of whose edges of a cycle $C_m$ are common to some other distinct cycles. Such a graph structure is called a {\em floral structure} of cycles in $G$. If $G$ itself is a floral structure of cycles, then it is called a {\em floral graph}.

Let $G$ be a floral graph in which $C_k$ be the cycle, some of whose edges are common to $C_{m_1},C_{m_2},C_{m_3},\ldots,C_{m_l}$. Then, each $C_{m_i}$ is called a {\em petal} of $G$ and $C_k$ is called the {\em nucleus} of $G$.

If the petals of a floral graph $G$  have no common edges in between them, such a graph is known as a {\em detached floral graph}. If the petals of a floral graph have some common edges in between them, such a graph is known as a {\em attached floral graph}.}
\end{definition}

Figures \ref{G-DFG} illustrates a detached floral graph and  Figures \ref{G-AFG} illustrates an attached floral graph.

\begin{figure}[h!]
\begin{center}
\begin{subfigure}[b]{0.4\textwidth}
\includegraphics[scale=0.35]{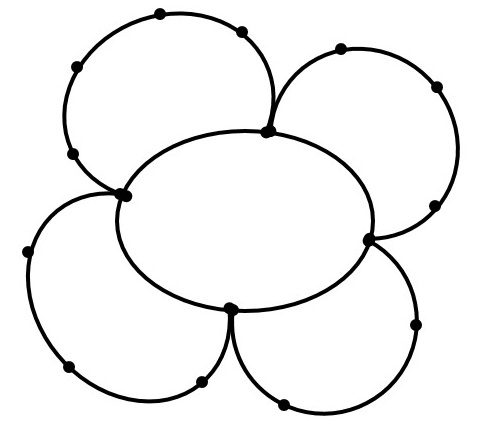}
\caption{\small \sl A detached floral graph.}\label{G-DFG}
\end{subfigure}
\qquad
\begin{subfigure}[b]{0.4\textwidth}
\includegraphics[scale=0.35]{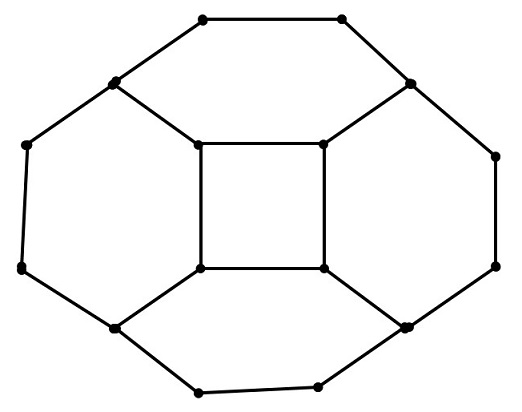}
\caption{\small \sl An attached floral graph.}\label{G-AFG}
\end{subfigure}
\caption{}
\end{center}  
\end{figure}

\begin{definition}{\rm
If all petals in a floral graph $G$ are of same parity, then $G$ is called a {\em like-petal graph} and if petals are of different parities, then $G$ is called a {\em mixed-petal graph}. If the petals in a mixed floral graph $G$ are alternately odd and even, then $G$ is called {\em alternating-petal graph}.

If the nucleus and petals of a floral graph $G$ are of same parity, then $G$ is called a {\em connatural graph}.}
\end{definition}

\noindent Figure \ref{G-CFG-E} illustrates a connatural floral graph of even cycles and Figure \ref{G-CFG-O} illustrates a connatural floral graph of odd cycles.

\begin{figure}[h!]
\begin{center}
\begin{subfigure}[b]{0.40\textwidth}
\includegraphics[scale=0.35]{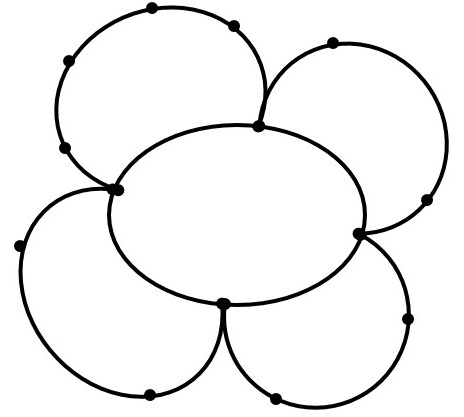}
\caption{\small \sl A connatural graph of even cycles.}\label{G-CFG-E}
\end{subfigure}
\qquad
\begin{subfigure}[b]{0.40\textwidth}
\includegraphics[scale=0.35]{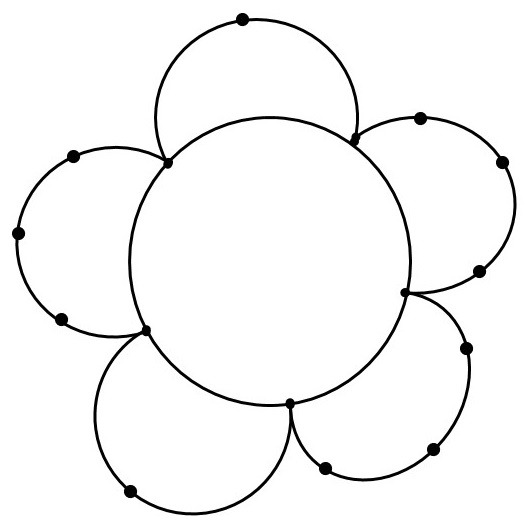}
\caption{\small \sl A connatural graph of odd cycles.}\label{G-CFG-O}
\end{subfigure}
\caption{}
\end{center}  
\end{figure}

\begin{remark}{\rm
All connatural graphs are like-petal graphs, but, the converse need not be true.}
\end{remark} 

Figure \ref{G-CFG} depicts a like-petal graph but not a connatural graph.

\begin{figure}[h!]
\centering
\includegraphics[width=0.35\textwidth]{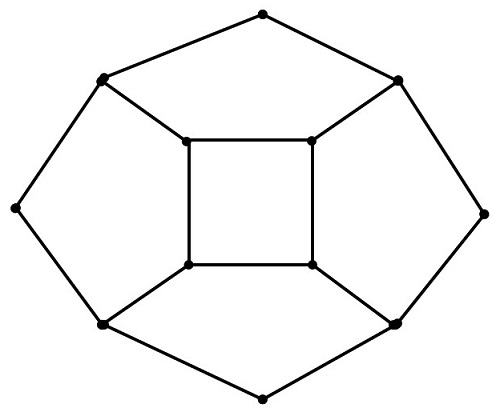}
\caption{\small \sl A like-petal graph that is not  connatural.}\label{G-CFG}
\end{figure}

In this section, we discuss about the sparing number of floral graphs. The floral graph $G$ mentioned in this section is like-petal, unless mentioned otherwise. 

If the nucleus and all petals of $G$ are even cycles, then the sparing number $G=C_k\cup(\displaystyle{\bigcup_{i=1}^l}C_{m_i})$ is $0$. Hence assume that all cycles of $G$ are not even. Let length of $C_k$ be $k$. Then,  $k\ge l$. 

\noindent {\bf Case-1:} Let nucleus and petals be all odd. Then, all $l+1$ cycles must have at least one odd cycle. 

Assume that $G$ is a detached floral graph. If $l$ is odd, then label the vertices of $G$ in such a way that the vertices (and hence the edges) of $C$ that are common to the cycles $C_{m_i},1\le i\le l$ are mono-indexed and the vertices of each $C_{m_i},1\le i\le l$ not common to $C_k$ have distinct non-singleton sets and distinct singleton sets (which are not used for labeling the vertices of any other cycles) as their set-labels. Then, the number of mono-indexed edges in $G$ under this labeling is $l$.

If $l$ is even, then $l<r$.  Label the vertices of $G$ as above such that  the edges of $C$ that are common to the cycles $C_{m_i},1\le i\le l$ are mono-indexed and no other edge $C_{m_i}$ is mono-indexed. Since $l$ is even, at least one more edge of $C_k$ must mono-indexed. Then, the number of mono-indexed edges in $C_k$ under this labeling is $l+1$. 

If $G$ is an attached floral graph, Then, then $\displaystyle{\bigcup_{i=1}^l}C_{m_i}$ is a entwined graph. If $l$ is odd, then for $1\le i\le \frac{(l-1)}{2}$, label the vertices of $G$ in such a way that the edge common to $C_{m_{2i-1}}$ and $C_{m_{2i}}$  and the edge common to $C_k$ and $C_{m_l}$ are mono-indexed. Therefore, the minimum number of mono-indexed vertices in $G$ is  $\frac{(l-1)}{2}+1 = \frac{(l+1)}{2}$. If $l$ is even, then clearly $l\le k$. As mentioned above, for $1\le i\le \frac{l}{2}$, label the vertices of $G$ in such a way that the edge common to $C_{m_{2i-1}}$ and $C_{m_{2i}}$  and the edge common to $C$ and $C_{m_l}$ and one edge of $C_k$, that is not common to any petal of $G$, are mono-indexed. Hence, the minimum number of mono-indexed edges in $G$ is $\frac{l}{2}+1=\frac{(l+2)}{2}$.

\noindent {\bf Case-2:} Let $C_k$ be an even cycle and all $C_{m_i},1\le i\le l$ be odd. Then, $C_k$ need not have any mono-indexed edge and each $C_{m_i}$ must have (at least) one mono-indexed edge. 

If $G$ is an attached floral graph, then $\displaystyle{\bigcup_{i=1}^l}C_{m_i}$ forms a entwined graph. Then by Theorem \ref{T-SNOCG}, $G$ has $\lfloor \frac{l+1}{2}\rfloor$ mono-indexed edges. 

If $G$ is a detached floral graph, then at least one edge of each $C_{m_i},1\le i \le l$ must be mono-indexed. Therefore, $G$ has $l$ mono-indexed edges.

\noindent {\bf Case-3:} Let $C_k$ be an odd cycle and all $C_{m_i},1\le i\le l$ are even. Then, $C_k$ must have (at least) one mono-indexed edge and no $C_{m_i}$ need to have a mono-indexed edge.

If $l<k$, then label the vertices of $G$ in such a way that an edge of $C_k$, that is not common to any $C_i$, is mono-indexed. Hence, $G$ has at least one mono-indexed edges.

If $l=k$, then label the vertices of $G$ in such a way that an edge of $C_k$ is mono-indexed. This edge is common to an even cycle, say $C_{m_j}$. Since an even cycle must have even number of mono-indexed edges, $C_{m_j}$  must have one more mono-indexed edge. Hence, $G$ has at least $2$  mono-indexed edges.

%%%%%%%%%%%%%%%%%%%%%%%%%%%%%%%%%%%%%%%%%%%%%%%%%%%%%%%%%%%%%%%%

%%%%%%%%%%%%%%%%%%%%%%%%%%%%%%%%%%%%%%%%%%%%%%%%%%%%%%%%%%%%%%%%%%%%%%%%%%%%%%%%%%%%%%%%%%%%%%%%%%%%%%%%%%%%%%%%%%%%%%%%%%%%%%%%%%%

\section {Conclusion}

%%%%%%%%%%%%%%%%%%%%%%%%%%%%%%%%%%%%%%%%%%%%%%%%%%%%%%%%%%%%%%%%%%%

The admissibility of weak IASI various other graph operations  and the corresponding sparing numbers are open. Listed below are some of the identified problems in this area.

\begin{problem}
Determine the admissibility of weak IASI by alternating-petal graphs and their sparing numbers.
\end{problem}

\begin{problem}
Determine the admissibility of weak IASI by non-alternating mixed petal graphs and their sparing numbers. 
\end{problem}

\begin{problem}
Determine the admissibility of weak IASI by different graph products and their sparing numbers.
\end{problem}

More properties and characteristics of weak and strong IASIs, both uniform and non-uniform, are yet to be investigated. We have formulated some conditions for some graph classes to admit weak and strong IASIs. The problems of establishing the necessary and sufficient conditions for various graphs and graph classes to have certain IASIs are open.

%%%%%%%%%%%%%%%%%%%%%%%%%%%%%%%%%%%%%%%%%%%%%%%%%%%%%%%%%%%%%%%%

%%%%%%%%%%%%%%%%%%%%%%%%%%%%%%%%%%%%%%%%%%%%%%%%%%%%%%%%%%%%%%%%%%%

\end{document}